\documentclass[notitlepage]{amsart}

\newtheorem{theorem}{Theorem}[section]
\newtheorem{corollary}[theorem]{Corollary}
\newtheorem{lemma}[theorem]{Lemma}
\newtheorem{proposition}[theorem]{Proposition}

\theoremstyle{definition}
\newtheorem{definition}[theorem]{Definition}

\theoremstyle{remark}
\newtheorem{remark}[theorem]{Remark}
\newtheorem{example}[theorem]{Example}

\usepackage{amsthm}           	    	              	%theorem environments 
\usepackage{amsmath}            	    	            %math features
\usepackage{amssymb}                	    	        %math symbols
\usepackage{enumerate}									%styles enumerate
	
\usepackage{csquotes} 		                          	%Quotes

\usepackage[T1]{fontenc}                                %font encoding

\usepackage[a4paper]{geometry}                          %margins
\usepackage{hyperref}

\usepackage{tikz-cd}                                    %tikz commutative diagrams
\usepackage[all]{xy}                                    %xymatrix
\usepackage{pgf,tikz}
\usepackage{mathrsfs}
\usetikzlibrary{arrows}

\newcommand{\End}{\mathrm{End}}
\newcommand{\Hom}{\mathrm{Hom}}
\newcommand{\Ext}{\mathrm{Ext}}
\newcommand{\Fac}{\mathrm{Fac}}
\newcommand{\Sub}{\mathrm{Sub}}
\newcommand{\proj}{\mathrm{proj}}

\newcommand{\F}{\mathcal{F}}

\newcommand{\X}{\mathcal{X}}
\newcommand{\Y}{\mathcal{Y}}

\newcommand{\f}{\mathrm{f}}
\newcommand{\modu}{\mathrm{mod}}
\newcommand{\Mat}{\mathrm{Mat}}
\newcommand{\udim}{\underline{\mathrm{dim}}}
\newcommand{\CoKer}{\mathrm{CoKer}}

\newcommand{\add}{\mathrm{add}}

\renewcommand{\t}{\mathrm{t}}

\newcommand{\rep}[1]{%
  {%
    \tiny%
    \begin{matrix}%
      #1%
    \end{matrix}%
  }%
}

\title{Stratifying systems and $g$-vectors}
\author{Octavio Mendoza, Corina S\'aenz and Hipolito Treffinger}
\begin{document}

\maketitle

\begin{abstract}
    In this paper we study the Cartan matrix associated to the Ext-projective stratifying system induced by a basic and $\tau$-rigid object $M$ in $\modu\,(A)$ by means of the $g$-vectors of the indecomposable direct summands of $M$.
   In particular we show that the Cartan group of a stratifying system associated to a $\tau$-rigid module can be calculated directly using these vectors.
    Moreover we characterise the stratifying systems coming from $\tau$-rigid modules that have a diagonal Cartan matrix.
\end{abstract}

\section{Introduction}

The notion of an Ext-injective stratifying system was introduced in \cite{Erdmann2003} as an axiomatisation of the standard objects of the standardly stratified algebras defined in \cite{Dlab}.
It was quickly realised that Ext-injective stratifying systems are equivalent to the notions of Ext-projective stratifying systems \cite{Marcos2005a} and what is now known simply as stratifying systems \cite{Marcos2004}.
Since then, stratifying systems were deeply studied in \cite{E05, HP06, Li13, MaMeSa06, Mes16, Sa18, SK09}. Throughout this paper, $A$ will denote to a basic finite dimensional  $k$-algebra over an algebraically closed filed $k.$ We will work on the abelian category $\modu(A)$ of all the finitely generated left $A$-modules. Given a positive integer $t,$ we denote by $[1,t]$  the set of all the natural numbers between $1$ and $t,$ and  it is also considered the natural order $\leq$ on the set $[1,t].$
\

The notion of  Ext-projective stratifying system in $\modu(A)$ is defined in \cite{Marcos2005a}. We recall this notion as follows.

\begin{definition}\cite[Definition 2.1]{Marcos2005a}\label{def:epss}
An Ext-projective stratifying system in $\modu(A),$ of size $t,$ is a triple $(\Theta,\underline{Q}, \leq ),$ where $\Theta:=\{\Theta(i)\}_{i=1}^t$ is a family of non-zero $A$-modules, $\underline{Q}:=\{Q(i)\}_{i=1}^t$ is a family of indecomposable $A$-modules and $\leq$ is the natural order on $[1,t]$ satisfying the following conditions:
\begin{enumerate}
\item[(a)] $\mathrm{Hom}_A(\Theta(j),\Theta(i))=0$ for $j> i;$
\item[(b)] for each $i\in [1,t]$, there is an exact sequence 
$0\to U(i)\xrightarrow{\alpha_i} Q(i)\xrightarrow{\beta_i}\Theta(i) \to 0$
in $\modu\,(A)$ such that  $U(i)$ is filtered by the set $\{\Theta(j): j>i\}$;
\item[(c)] $\Ext^1_A(Q,X)=0$ where $Q:=\bigoplus_{i=1}^tQ(i)$ and $X$ is an $A$-module filtered by $\{\Theta(i)\}_{i=1}^t$.
\end{enumerate}
\end{definition}

Although it has been shown that stratifying systems are objects having very nice properties, until very recently there was no systematic method to construct stratifying systems in the module category of a given algebra $A$. 
This problem has been addressed in \cite{Mendoza2019} and \cite{Treffinger2021}  using tools from $\tau$-tilting theory and higher homological algebra, respectively.

It was shown in \cite{MMS2019} that it can be associated to every Ext-projective stratifying system $(\Theta,\underline{Q}, \leq )$ two invariants $C_\Theta$ and $G_\Theta$, known as the \textit{Cartan matrix} and the \textit{Cartan group} of $\Theta$, respectively (see Section~\ref{sec:stratifying}). %, which is a squared $t\times t$-matrix where the $(i,j)$-coordinate  of $C_\Theta$ is $[C_\Theta]_{i,j}:=\dim_k\Hom_A(Q(i),\Theta(j)).$
In this paper,  we study the Cartan matrix $C_{\Delta_M}$ and the Cartan group $G_{\Delta_M}$  for a stratifying system $\Delta_M$ associated to a (basic) 
$\tau$-rigid module $M\in\modu(A).$
In particular, we show that the matrix $C_{\Delta_M}$ can be obtained using the  $g$-matrix $G^M$ (formed by the $g$-vectors of the indecomposable direct summands of $M$), the $M$-standard matrix $S^M$ (formed by the dimension vectors of the $M$-standard modules $\Delta_M$) and the $M$-residual matrix $R^M$ (see Section~\ref{sec:mainresults}). This result can be summarized as follows, for the complete version, see 
Theorem~\ref{MTM} and Proposition \ref{MinFD}.

\begin{theorem}
Let $M\in\modu\,(A)$ be basic and $\tau$-rigid  with a TF-admissible decomposition $M= \bigoplus_{i=1}^t M_i,$ and let $(\{\Delta_M\}_{i=1}^t, \underline{Q}, \leq )$ be its Ext-projective stratifying system. Then,  
$$C_{\Delta_M}=(G^M)^{tr}\,S^M+R^M\quad\text{and}\quad 
|G_{\Delta_M}| = \prod_{i=1}^t [(G^M)^{tr}\,S^M]_{i,i}$$
Moreover, if $M$ is filtered by the family $\Delta_M,$ then $R^M=0$ and $Q(i)\simeq M_i$ $\forall\,i.$
\end{theorem}

Then we apply this theorem to characterise algebras $A$ and stratifying systems in $\modu(A)$ coming from $\tau$-tilting modules having a diagonal Cartan matrix, generalising \cite[Theorem~5.9]{MMS2019}. In what follows, we summarize this result. For more details, see Theorem~\ref{thm:diagonal}.

\begin{theorem}
For a basic $\tau$-tilting module $M\in\modu\,(A)$ with a TF-admissible decomposition $M= \bigoplus_{i=1}^n M_i$ such that 
$M\in\F(\Delta_M),$  the following statements are equivalent.
\begin{itemize}
\item[$\mathrm{(a)}$] The matrix $C_{\Delta_M}$ is diagonal.
\item[$\mathrm{(b)}$] $\Hom_A(M_i, M_j)=0$ for $i<j.$
\item[$\mathrm{(c)}$] $\Hom_A(M_i, \Delta_M(j))=0$ for $i<j.$
\end{itemize}
Moreover, if one of the above conditions holds true, then $M\simeq {}_AA$ and 
$A$ is a weakly triangular algebra.
\end{theorem}

\smallskip
The paper is organised as follows. 
In Section~\ref{sec:setting} we establish the setting and notation in which we work and we recall some basic representation theory.
Then, in Section~\ref{sec:stratifying} we recall some theory on stratifying systems and we prove some preparatory results. 
Later, in Section~\ref{sec:tau-tilting} we give a brief overview of the necessary results on $\tau$-tilting theory that we need in  
Section~\ref{sec:mainresults}  where we state and prove the main results of this paper. Finally, in Section~\ref{Examples} we give some examples illustrating Theorem~\ref{MTM}.

\subsection*{Ackowledgements}
Most of the mathematical content of this paper was performed while the third named author visited the other two authors at the UNAM in M\'exico in February 2020. 
We hope that in the near future the sanitary conditions worldwide will allow for more in person collaborations.
\

The authors thank the research project PAPIIT-Universidad Nacional Aut\'onoma de M\'exico IN100520.
\
H.~Treffinger was partially funded by the Deutsche Forschungsgemeinschaft (DFG, German Research Foundation) under Germany's Excellence Strategy Programme -- EXC-2047/1 -- 390685813.
H.~Treffinger is also supported by the European Union’s Horizon 2020 research and innovation programme under the Marie Sklodowska-Curie grant agreement No 893654.
He was also partially supported by the EPSRC through the Early Career Fellowship, EP/P016294/1.
H.~Treffinger would like to thank the Isaac Newton Institute for Mathematical Sciences, Cambridge, for support and hospitality during the programme Cluster Algebras and Representation Theory where work on this paper was undertaken. 
This work was supported by EPSRC grant no EP/K032208/1 and the Simons Foundation.

\section{Setting and background}\label{sec:setting}

For $M\in\modu\,(A),$ we denote by $rk(M)$ the number of pairwise non-isomorphic indecomposable direct summands of $M.$ Finally, we say that $M\in\modu\,(A)$ is {\it basic} if there is a decomposition $M=\bigoplus_{i=1}^t M_i,$ with $M_i$ indecomposable for every $i$ and $M_i\not\simeq M_j$ if $i\neq j.$ 

Given an algebra $A$, we denote by $M\twoheadrightarrow N$ an epimorphism from $M$ to $N$ in $\modu\,(A);$ and for each $M\in\modu\,(A),$ we consider the class 
\begin{center} $\Fac(M):=\{X \in\modu\,(A)\; :\;\exists\, M^{n} \twoheadrightarrow X \text{ for some $n\in\mathbb{N}$}\}.$\end{center}
Dually, the arrow $\hookrightarrow$ stands for a monomorphism in $\modu\,(A),$ and we have the class 
$$\Sub(M):=\{X \in\modu\,(A)\; :\;\exists\, X \hookrightarrow M^{n} \text{ for some $n\in\mathbb{N}$}\}.$$
For any subclass $\X\subseteq\modu\,(A),$ we have the right $\Hom$-perpendicular complement of $\X$
\begin{center}
 $\X^\perp:=\{M\in\modu\,(A)\;:\;\Hom_A(-,M)|_\X)=0\}.$   
\end{center}
Dually,  we have the left $\Hom$-perpendicular complement ${}^\perp\X$ of $\X.$ For a single element class $\X=\{M\},$ we just write $M^\perp$ and ${}^\perp M.$ 
It is said that $N\in\modu\,(A)$ admits an {\it $\X$-filtration} if there is a chain of submodules 
$0=M_0\subseteq M_1\subseteq\ldots\subseteq M_{n-1}\subseteq M_n=M$ satisfying that each quotient 
$M_i/M_{i-1}\simeq X_i\in\X.$ 
The class of all the  $N\in\modu\,(A)$ which admits an $\X$-filtration will be denoted by $\F(\X).$
\medskip

Let $\X\subseteq\modu\,(A)$ and $M\in\modu\,(A).$ 
A morphism $f:X\to M$ in $\modu\,(A)$ is called $\X$-precover (or right $\X$-approximation) of $M$ if $X\in\X$ and $\Hom_A(X',f):\Hom_A(X',X)\to\Hom_A(X',M)$ is surjective, for any $X'\in\X$. 
Moreover, if the equality $fh=f$ holds true only for an automorphism $h:X\to X,$ it is said that the $\X$-precover is an $\X$-cover. 
The class $\X$ is precovering (or contravariantly finite) if any $N\in\modu\,(A)$ admits an $\X$-precover. 
Dually, we have the notions of $\X$-preenvelope (or left $\X$-approximation), $\X$-envelope and preenveloping (or covariantly finite) class. 
If $\X$ is precovering and preenveloping,  it is said  that $\X$ is functorially finite.
\medskip

We recall that a pair $(\X, \Y)$ of full subcategories in $\modu\,(A)$ is a \textit{torsion pair} if $\X={}^\perp\Y$ and $\Y=\X^\perp.$
Given a torsion pair $(\X, \Y)$ in $\modu\,(A),$ it is said that $\X$ is a torsion class and $\Y$ is a torsion free class. 
Moreover, in this case,  it can be shown that $\X$ is closed under quotients and extensions, while $\Y$ is closed under submodules and extensions.
It is also well known that, for every subclass $\X\subseteq\modu\,(A)$ which is closed under quotients and extensions, there exists a subclass $\Y\subseteq\modu\,(A)$ such that $(\mathcal{X}, \mathcal{Y})$ is a torsion pair in $\modu\,(A).$
In particular, if $\mathcal{X}=\Fac(M)$ for some $M\in\modu\,(A),$ then $\mathcal{Y} = M^{\perp}.$

It is also well known that  each torsion pair $(\X,\Y)$ in $\modu\,(A)$ has associated two additive functors 
$\t,\f: \modu\,(A) \to \modu\,(A).$ The {\it torsion functor} $\t$ is a subfunctor of the identity functor $1_{\modu\,(A)},$ 
and the {\it torsion free functor} $\f$ is equal to the quotient functor $1/\t.$  
Moreover, for every $M \in \modu\,(A),$ there exists the so-called canonical short exact sequence
\begin{center}
$0\to \t(M)\to M\to \f(M)\to 0,$
\end{center}
where $\t(M)\in\mathcal{X}$ and $\f(M)\in\mathcal{Y}$ and such exact sequence is unique up to isomorphisms.

Let $\X$ be a full subcategory of $\modu\,(A).$ 
It is said that $X\in\X$ is {\it Ext-projective} in $\X$ if $\Ext^1_A(X,-)|_\X=0.$ 
Moreover, if $\X$ is functorially finite and a torsion class, it is well known that there are only finitely many (up to isomorphisms) indecomposable Ext-projective modules in $\X$ and we denote by $\mathbb{P}(\X)$ their direct sum.

\section{Stratifying systems}\label{sec:stratifying}

The concept of stratifying system for module categories was first introduced in \cite{Erdmann2003}, and further studied in \cite{Marcos2004,Marcos2005a}.
The original definition is stated in Definition\;\ref{def:eiss}, but it was shown in \cite{Marcos2004} that this definition is equivalent to the following one.

\begin{definition}\cite[Characterisation 1.6]{Marcos2004}\label{def:stratsys}
A stratifying system of size $t,$ in $\modu\,(A),$ consists of a pair $(\Theta, \leq),$ where $\Theta := \{ \Theta(i)\}_{i=1}^t$ is a family of indecomposable objects in $\modu\,(A)$ and $\leq$ is the natural order on the set $[1,t]:=\{1, 2, \dots , t\}$ satisfying the following conditions:
\begin{enumerate}
   \item[$\mathrm{(a)}$] $\Hom_A(\Theta(j), \Theta(i))=0$ if $j > i$;
   \item[$\mathrm{(b)}$] $\Ext^1_A(\Theta(i), \Theta(j))=0$ if $j \geq i$.
\end{enumerate}
\end{definition}

Let $(\Theta, \leq)$ be a stratifying system in $\modu\,(A).$ The theory of stratifiying systems, developed in \cite{Erdmann2003, Marcos2004, Marcos2005a}, shows that  $\F(\Theta)$ is equivalent to the full subcategory $\F(\Delta)$ of $\modu\,(B)$, where $B$ is a standardly stratified algebra and $\F(\Delta)$ is the category of the $\Delta$-filtered objects, of key importance in the theory of quasi-hereditary and standardly stratified algebras. In order to determine the standardly stratified algebra $B$, it is necessary to consider the so-called \textit{Ext-projective} and \textit{Ext-injective} stratifying systems. The notion of Ext-projective stratifying system was given in Definition \ref{def:epss}.

\begin{definition}\cite[Definition 1.1]{Erdmann2003}\label{def:eiss}
An Ext-injective stratifying system in $\modu(A),$ of size $t,$ is a triple 
$(\Theta,\underline{Y}, \leq ),$ where $\Theta:=\{\Theta(i)\}_{i=1}^t$ is a family of non-zero $A$-modules, $\underline{Y}:=\{Y(i)\}_{i=1}^t$ is a family of indecomposable $A$-modules and $\leq$ is the natural order on $[1,t]$ satisfying the following conditions:
\begin{enumerate}
\item[$\mathrm{(a)}$] $\mathrm{Hom}_A(\Theta(j),\Theta(i))=0$ for $j> i;$
\item[$\mathrm{(b)}$] for each $i\in [1,t]$, there is an exact sequence $0\to \Theta(i)\xrightarrow{\lambda_i} Y(i)\xrightarrow{\gamma_i}  Z(i)\to 0$ in 
$\modu\,(A)$ 
such that  $Z(i)\in\mathcal{F}(\{\Theta(j): j<i\});$
\item[$\mathrm{(c)}$] $\Ext^1_A(-,Y)|_{\F(\Theta)}=0,$ for $Y:=\bigoplus_{i=1}^t Y(i).$
\end{enumerate}
\end{definition}

Given a stratifying system $(\Theta, \leq),$ it is well known \cite{Erdmann2003, Marcos2004, Marcos2005a}  that there exist a unique (up to isomorphism) Ext-projective stratifying system $(\Theta,\underline{Q}, \leq )$ and an Ext-injective stratifying system $(\Theta,\underline{Y}, \leq ).$ 
Furthermore, for a given Ext-projective stratifying system $(\Theta,\underline{Q}, \leq ),$ we have that the pair 
$(\Theta, \leq )$ is a stratifying system. Moreover the same statement holds true for Ext-injective stratifying systems.

Let  $(\Theta,\{Q(i)\}_{i=1}^t,\leq)$ be an Ext-projective stratifying system in $\modu\,(A).$ By following \cite[Section 4]{MMS2019}, we consider the $\Theta$-Cartan matrix $C_\Theta\in\Mat_{t\times t}(\mathbb{Z}),$ naturally associated to this system, which is defined as follows: the $(i,j)$-coordinate  of $C_\Theta$ is $[C_\Theta]_{i,j}:=\dim_k\Hom_A(Q(i),\Theta(j)).$

The main idea in \cite{MMS2019} behind the introduction of the Cartan matrix $C_\Theta,$ for an Ext-projective stratifying system $(\Theta,\{Q(i)\}_{i=1}^t, \leq ),$ is that we can use it to construct a finite abelian group $G_\Theta$, known as the \emph{Cartan Group} of $\Theta$ which is defined as $G_\Theta:=\CoKer(C_\Theta),$ where 
the matrix $C_\Theta$ is identified with the $\mathbb{Z}$-linear map $\mathbb{Z}^t \to \mathbb{Z}^t,\;X\mapsto C_\Theta\,X.$ It is shown in \cite{MMS2019}, that 
$|G_\Theta| = 1$ if and only if the standardly stratified algebra $\End_A (Q)^{op}$ is quasi-hereditary. Then the size $|G_\Theta|$ of $G_\Theta$ give us a measure of how far is $\End_A(Q)^{op}$ from being quasi-hereditary.

\begin{lemma} \label{LCMTheta} \cite{Marcos2005a, MMS2019, MMSZ} Let $(\Theta,\{Q(i)\}_{i=1}^t,\leq)$ be an Ext-projective stratifying system in $\modu\,(A).
$ Then, $C_\Theta$ is an upper triangular matrix and $[C_\Theta]_{i,i}=\dim_k(\End_A(\Theta(i)))$ $\forall\,i\in[1,t].$
Moreover $|G_\Theta|= \prod_{i=1}^n [C_\Theta]_{i,i}.$
\end{lemma}

\begin{proof} It  follows from \cite[Lemma 2.5 (b)]{MMSZ} and \cite[Lemma 2.6 (b)]{Marcos2005a}.
\end{proof}

Let $(\Theta,\{Q(i)\}_{i=1}^t,\leq)$ be an Ext-projective stratifying system  in $\modu\,(A).$ In what follows, for each $i\in[1,t],$ we consider 
the exact sequence 
$\eta_i:\; 0 \rightarrow U(i) \xrightarrow{\alpha_i} Q(i) \xrightarrow{\beta_i} \Theta(i) \rightarrow 0$ associated to each 
$\Theta(i),$ see Definition \ref{def:epss} (b), and the map 
$$\partial_{i,j}:\Hom_A(U(i),\Theta(j))\to \Ext^1_A(\Theta(i),\Theta(j)),\;f\mapsto f\cdot \eta_i,$$ where $f\cdot \eta_i$ is the push-out of 
$f$ and $\eta_i.$

\begin{proposition}\label{PEDiagCtheta} For an Ext-projective stratifying system $(\Theta,\{Q(i)\}_{i=1}^t,\leq)$ in $\modu\,(A),$ 
the following statements are equivalent.
\begin{itemize}
\item[$\mathrm{(a)}$] The matrix $C_\Theta$ is diagonal.
\item[$\mathrm{(b)}$] $\Hom_A(Q(i), Q(j))=0$ for  $i < j.$
\item[$\mathrm{(c)}$] $\Hom_A(Q(i), \Theta(j))=0$ for  $i < j.$
\item[$\mathrm{(d)}$] For $i<j,$ we have that $\Hom_A(\Theta(i),\Theta(j))=0$ and 
$$\partial_{i,j}:\Hom_A(U(i),\Theta(j))\to \Ext^1_A(\Theta(i),\Theta(j)),\;f\mapsto f\cdot \eta_i$$ is an isomorphism.
\end{itemize}
Moreover, if one of the above conditions holds true, then $\Hom_A(Q(i),U(j))=0$ for $i<j.$
\end{proposition}
\begin{proof} (a) $\Leftrightarrow$ (c): It follows from Lemma \ref{LCMTheta}.
\

(b) $\Rightarrow$ (c): Let $\Hom_A(Q(i), Q(j))=0$  for  $i < j.$  We prove now that $\Hom_A(Q(i), \Theta(j))=0.$
If $Q(j)=\Theta(j)$ there is nothing to prove.
Otherwise, take $f: Q(i) \to \Theta(j)$ and consider the exact sequence 
$$\eta_j:\quad 0 \rightarrow U(j) \xrightarrow{\alpha_j} Q(j) \xrightarrow{\beta_j} \Theta(j) \rightarrow 0$$ associated to $\Theta(j),$ 
where $U(j)$ is nonzero and $U(j) \in \F(\{\Theta(k) : k>j\}).$
Since $Q(i)$ is Ext-projective in $\F(\Theta)$,  there exist $f': Q(i) \to Q(j)$ such that $f'\beta_j = f$.
By hypothesis, $f'=0$. Hence $f=0$.
\

(c) $\Rightarrow$ (b): Let $\Hom_A(Q(i), \Theta(j))=0$ for $i<j.$ We carry on the proof by reverse induction on $j.$  
If $j=n$ then $\Theta(n)=Q(n)$ and thus (b) follows. Therefore, we can assume that $j<n.$ Let us prove that $\Hom_A(Q(i), Q(j))=0.$ In order to do that, take a map $g : Q(i) \to Q(j)$ and consider the exact sequence $\eta_j.$
Since $g \beta_j =0$, there exist $h: Q(i) \to U(j)$ such that $\alpha_j h = g.$
Moreover,  since $U(j) \in \F(\{\Theta(k) : k>j\})$, \cite[Proposition 2.10]{Marcos2005a} implies the existence of an exact sequence  
$$ 0 \rightarrow N \to Q_0(U(j)) \xrightarrow{\epsilon_{U(j)}} U(j) \rightarrow 0 $$
in $\F(\Theta),$ where $Q_0(U(j)) \in \add\left( \bigoplus_{k>j} Q(k) \right)$.
Once again, by using that $Q(i)$ is Ext-projective in $\F(\Theta),$  there exist $h': Q(i) \to Q_0(U(j))$ such that $\epsilon_{U(j)}h' = h$.
By induction, $h'=0$.
Therefore $g=\alpha_j\epsilon_{U(j)}h'=0$.
\

(c) $\Leftrightarrow$ (d): Let $i<j.$ By applying the functor $\Hom_A(-,\Theta(j))$ to the exact sequence $\eta_i,$ we get the following exact sequence
$$0\to \Hom_A(\Theta(i),\Theta(j))\to \Hom_A(Q(i),\Theta(j))\to \Hom_A(U(i),\Theta(j))\to \Ext^1_A(\Theta(i),\Theta(j))\to 0$$
Thus, the equivalence between (c) and (d) follows from the above exact sequence.
\

Finally, assume that (b) holds true. Then, by applying the functor $\Hom_A(Q(i),-)$ to $\eta_j$ we get that $\Hom_A(Q(i),U(j))=0$ for $i<j.$

\end{proof}

\begin{corollary}\label{CEDiagCtheta}   For an Ext-projective stratifying system $(\Theta,\{Q(i)\}_{i=1}^t,\leq)$ in $\modu\,(A),$ such that 
the matrix $C_\Theta$ is diagonal, the following statements are equivalent.
\begin{itemize}
\item[$\mathrm{(a)}$] $[U(i):\Theta(j)]=0$ for $i<j.$
\item[$\mathrm{(b)}$] $\Ext^1_A(\Theta(i),\Theta(j))=0$ $\forall\,i,j.$
\item[$\mathrm{(c)}$] $Q(i)\simeq \Theta(i)$ as $A$-modules $\forall\,i.$
\item[$\mathrm{(d)}$] $U(i)=0$ $\forall\,i.$
\end{itemize}
Moreover, if one of the above conditions holds true, then $\F(\Theta)=\add(Q)$ and $\End_A(Q)\simeq \times_{i=1}^t\,\End_A(Q(i))$ as $k$-algebras. 
\end{corollary}
\begin{proof} By Proposition \ref{PEDiagCtheta}, we get that $\Hom_A(\Theta(i),\Theta(j))=0$ $\forall\,i\neq j,$ and moreover 
$\partial_{i,j}:\Hom_A(U(i),\Theta(j))\xrightarrow{\simeq} \Ext^1_A(\Theta(i),\Theta(j))$ for $i<j.$
\

(a) $\Rightarrow$ (b): Let $i,j\in [1,t].$ By \cite[Lemma 2.6 (b)]{Marcos2005a}, we know that $\Ext^1_A(\Theta(i),\Theta(j))=0$ for $i\geq j.$ Hence, we can assume that $i<j.$ Since $U(i)\in \F(\Theta(l)\;:\; l>i),$ by (a) we have that $\Hom_A(U(i),\Theta(j))=0.$ Therefore 
$\Ext^1_A(\Theta(i),\Theta(j))=0$ for $i< j.$
\

(b) $\Rightarrow$ (c):  Let $i\in[1,t].$ Then, by (b), we get that the exact sequence $\eta_i:\quad 0 \rightarrow U(i) \xrightarrow{\alpha_i} Q(i) \xrightarrow{\beta_i} \Theta(i) \rightarrow 0$ splits. Thus $Q(i)\simeq \Theta(i)$ as $A$-modules since $Q(i)$ and $\Theta(i)$ are indecomposable (see \cite[Corollary 2.13]{Marcos2005a}).
\

(c) $\Rightarrow$ (d): From (c) and the exact sequence $\eta_i,$ we have that $\dim_k U(i)=0$ and hence $U(i)=0.$
\

(d) $\Rightarrow$ (a): It is trivial.

Finally,  If one of the above conditions holds true, then we get that 
$Q(i)\simeq \Theta(i)$ as $A$-modules $\forall\,i.$ Hence $\Hom_A(Q(i),Q(j))=0$ for $i\neq j$ and thus $\End_A(Q)\simeq \times_{i=1}^t\,\End_A(Q(i))$ as $k$-algebras. 
\end{proof}

\section{$\tau$-tilting theory}\label{sec:tau-tilting}

Most of the results of this paper are proven using the tools provided by $\tau$-tilting theory \cite{AIR}.
We now give a brief summary of the definitions and results on $\tau$-tilting theory that will be needed later. 
For a broader survey on this topic, the reader is encouraged to see \cite{TreffingerSurvey}.

\begin{definition}\cite[Definition 0.1]{AIR}
Let  $M\in\modu\,(A).$  If $\Hom_A(M, \tau M)=0,$ it is said that $M$ is \textit{$\tau$-rigid}. A $\tau$-rigid module $M$ is \textit{$\tau$-tilting} if $rk(M)=rk({}_AA).$ A pair $(M,P)$ is called $\tau$-rigid if $P\in\modu\,(A)$ is projective, $M$ is $\tau$-rigid and $\Hom_A(P,M)=0.$
Finally, the $\tau$-rigid pair $(M,P)$ is $\tau$-tilting if $rk(M)+rk(P)=rk({}_AA).$
\end{definition}

One of the main features of $\tau$-tilting theory is that all functorially finite torsion classes in $\modu\,(A)$ can be described by using $\tau$-tilting pairs, as stated in the following result that has been taken from \cite[Theorem 2.7]{AIR} and \cite[Theorem 5.10]{Auslander1981}.

\begin{theorem}\label{thm:fftors}
For any algebra $A,$ there is a well defined function 
\begin{center} $\Phi: \mathrm{\tau\text{-}rig}(A) \to \mathrm{f\text{-}tors}(A),\;M\mapsto \Fac(M),$
\end{center} from $\tau$-rigid basic pairs to functorially finite torsion classes in $\modu\,(A).$ 
Moreover, $\Phi$ is a bijection if we restrict it to the class $\mathrm{s\tau\text{-}tilt}(A)$ of $\tau$-tilting pairs, and in this case $\Phi^{-1}(\X)=\mathbb{P}(\X).$
\end{theorem}

In \cite{Mendoza2019}, it was shown how to produce a stratifying system from a basic $\tau$-rigid module $M\in\modu\,(A).$ We start by recalling the notion of torsion free admissible decomposition.

\begin{definition} \cite[Definition 3.1]{Mendoza2019}  Let $M\in\modu\,(A)$ be basic and $\tau$-rigid. A decomposition $M=\bigoplus_{i=1}^t M_i$ as the direct sum of indecomposables $A$-modules is called {\it torsion free admissible} (TF-admissible, for short) if $M_i \not \in \Fac\left(\bigoplus_{j>i} M_j\right),$
for every $i\in[1,t].$
\end{definition}

The existence of such TF-admissible decompositions is supported by the following proposition proved in \cite{Mendoza2019}.

\begin{proposition}\cite[Proposition 3.2]{Mendoza2019}
Every basic and $\tau$-rigid $A$-module $M\in\modu\,(A)$ admits a TF-admissible decomposition.
\end{proposition}

The following result, proved in \cite{Mendoza2019}, tells us how to build a stratifying system $(\Delta_M,\leq)$ from a TF-admissible decomposition of $M\in\modu\,(A).$ The stratifying system $(\Delta_M,\leq)$ is known as the {\it $M$-standard stratifying system} associated to this TF-admissible decomposition.

\begin{theorem}\cite[Theorem 3.4]{Mendoza2019}
Let $M\in\modu\,(A)$ be a basic and $\tau$-rigid  with a TF-admissible decomposition 
$M= \bigoplus_{i=1}^t M_i,$ and let  $\f_k$ be the torsion free functor associated to the torsion pair 
$\left( \Fac(\bigoplus_{j \geq k}M_j), (\bigoplus_{j \geq k}M_j)^\perp  \right).$
Then, the family $\Delta_M := \{ \Delta_M(i) := \f_{i+1}(M_i) \}_{i=1}^t$ and the natural order on $[1,t]$ form  a stratifying 
system in $\modu\,(A)$ of size $t$. 
\end{theorem}

In this paper we study the relation between the homological properties of a  basic and $\tau$-rigid module $M\in\modu\,(A)$ and its induced $M$-stratifying system $(\Delta_M,\leq).$ In order to do that,  we need to introduce a vector with integer coordinates, known as the $g$-vector of $M$, associated to each module $M$ in $\modu\;(A).$ For defining such $g$-vector, we denote by $P(1),\ldots, P(n)$ the iso-classes of indecomposable projective $A$-modules.

\begin{definition}\cite[Section 5]{AIR}
Let $M\in\modu\,(A)$ and $P_1\to P_0\to M\to 0$ be the minimal projective presentation of $M,$ where $P_0=\bigoplus\limits_{i=1}^n P(i)^{a_i}$ and $P_1=\bigoplus\limits_{i=1}^n P(i)^{a'_i}.$
The $g$-vector of $M$ is set to be $g^M:=(a_1-a'_1, a_2-a'_2,\dots, a_n-a'_n)\in\mathbb{Z}^n.$
\end{definition}

Another useful vector associated with $M\in\modu\,(A)$ is the {\it dimension vector} $\udim(M)\in\mathbb{Z}^n$ whose $i$-th coordinate is the composition factor of the simple $S(i)$ in $M,$ where $S(i)$ is the simple top of $P(i).$ 
There is important homological information that arises when $g$-vectors and dimension vectors interplay, as shown in \cite{Auslander1985}.
For $u,v\in\mathbb{Z}^n,$ we denote by $\langle u, v\rangle$ the standard inner product $\sum_{i=1}^n u_iv_i$ of these vectors.

\begin{theorem}\cite[Theorem 1.4.(a)]{Auslander1985}\label{formula}
Let $M$ and $N$ be in $\modu\,(A).$ Then 
$$\langle g^{M},\udim(N)\rangle=\dim_k(\emph{Hom}_A(M,N))-\dim_k(\emph{Hom}_A(N,\tau_A M)).$$
\end{theorem}

\section{Main results}\label{sec:mainresults}

In this section we state and prove our main results. 
From now on, let  $n:=rk({}_AA)$ and fix some basic and $\tau$-rigid module $M$ in $\modu\,(A)$ with some TF-admissible decomposition 
$M= \bigoplus_{i=1}^t M_i.$
As we mentioned in the previous section, this determines a stratifying system $(\Delta_M,\leq)$ which, in turn, induces an Ext-projective stratifying system $(\Delta_M, \underline{Q}, \leq)$. 
Moreover, for each $i\in[1,t],$ we consider the following exact sequences:

(1) the canonical exact sequence
$$\varepsilon_i:\; 0\to  K(i) \xrightarrow{\iota_i} M_i \xrightarrow{\pi_i} \Delta_M(i) \rightarrow 0 $$
of $M_i$ with respect to the torsion pair $\left( \Fac(\bigoplus_{l>i} M_l), (\bigoplus_{l>i} M_l)^\perp \right);$   
\

(2) the exact sequence 
$$\eta_i:\; 0\to U(i)\xrightarrow{\alpha_i} Q(i)\xrightarrow{\beta_i}\Delta_M(i)\to 0$$
associated to each $\Delta_M(i),$ see Definition \ref{def:epss} (b).
We also consider the following integer matrices:
\

(3) the {\bf $G$-matrix} $G^M\in\Mat_{n\times t}(\mathbb{Z})$ of $M$ whose $i$-th column is $[G^M]^i:=(g^{M_i})^{tr};$
\

(4) the {\bf $M$-standard matrix} $S^M\in\Mat_{n\times t}(\mathbb{Z})$ whose $i$-th column is  
$[S^M]^i:=\udim(\Delta_M(i))^{tr};$
\

(5) the {\bf $M$-residual matrix} $R^M\in\Mat_{t\times t}(\mathbb{Z})$ whose $(i,t)$-coordinate is
$$[R^M]_{i,j}:=\dim_k\Hom_A(U(i),\Delta_M(j))-\dim_k\Hom_A(K(i),\Delta_M(j))$$

\begin{remark}
We note that in (3) we are incurring in an abuse of notation, since the $G$-matrices are only defined for $\tau$-tilting pairs, see \cite{DIJ, Fu2017, Treffinger2019}.
\end{remark}

An important lemma in the proof of our main theorem is the following.

\begin{lemma}\label{lem:M-theta}
For a basic and $\tau$-rigid module $M\in\modu\,(A)$ with a TF-admissible decomposition $M= \bigoplus_{i=1}^t M_i,$ and  the Ext-projective stratifying system $(\Delta_M, \underline{Q}, \leq)$ induced by $(\Delta_M,\leq),$ the following statements hold true.
\begin{itemize}
\item[$\mathrm{(a)}$] $\Hom_A(U(i),\Delta_M(j))=0$ and $\Hom_A(K(i),\Delta_M(j))=0$ if $i\geq j.$
\item[$\mathrm{(b)}$] We have the isomorphisms of $k$-vector spaces 
$$\Hom_A(M_i, \Delta_M(i)) \simeq \Hom_A(Q(i), \Delta_M(i))\simeq \End_A(\Delta_M(i))\quad\forall\,i\in[1,t].$$
\item[$\mathrm{(c)}$] For all $i,j\in[1,t],$ we have 
$$[C_{\Delta_M}]_{i,j}=\dim_k \Hom_A(M_i,\Delta_M(j))+[R^M]_{i,j}$$
\item[$\mathrm{(d)}$]  $\Hom_A(\Delta_M(j),\tau M_i)=0$ and $\Ext^1_A(M_i,\Delta_M(j))=0$  $\forall\,i,j.$
\end{itemize}
\end{lemma}
\begin{proof} Fix some $i\in[1,t].$ Then, we have the exact sequence
$$\eta_i:\; 0 \rightarrow U(i) \xrightarrow{\alpha_i} Q(i) \xrightarrow{\beta_i} \Delta_M(i) \rightarrow 0,$$
where $U(i)\in \F(\Delta_{M}(j) : j >i ),$ and the  canonical  exact sequence 
$$\varepsilon_i:\; 0 \rightarrow K(i) \xrightarrow{\iota_i} M_i \xrightarrow{\pi_i} \Delta_M(i) \rightarrow 0 $$
of $M_i$ with respect to the torsion pair $(\X_i,\Y_i):=\left( \Fac(\bigoplus_{l>i} M_l), (\bigoplus_{l>i} M_l)^\perp \right).$ Moreover, by \cite[Corallary 3.6(c)]{Mendoza2019}, we have that $\F(\Delta_M)$ is a subcategory of $\Fac(M).$

We know that $M_i$ is Ext-projective in $\Fac(M)$.
Since $\F(\Delta_M)\subseteq\Fac(M)$, we have in particular that $M_i$ is Ext-projective in $\F(\Delta_M)$.
Hence, we can complete the following commutative diagram
$$\xymatrix{
 0\ar[r]& K(i) \ar[r]^{\iota_i}\ar@{.>}[d]^{q} & M_i\ar[r]^{\pi_i}\ar@{.>}[d]^{p} & \Delta_M(i   ) \ar[r]\ar@{=}[d] &0 \\
 0\ar[r]& U(i)\ar[r]^{\alpha_i} & Q(i) \ar[r]^{\beta_i} &\Delta_M(i) \ar[r]&0. }$$ 
Note that the first square of the above diagram is the push-out of $q$ and $\alpha_i.$
Hence, there exist an  exact sequence 
$0 \xrightarrow{ } K(i) \xrightarrow{ } U(i)\bigoplus M_i \xrightarrow{ } Q(i) \xrightarrow{ } 0$
to which we apply the functor $\Hom_A(-,\Delta_M (j))$ to obtain the following one
$$(*):\;0 \xrightarrow{ } \Hom_A(Q(i),\Delta_M(j)) \to \Hom_A(U(i)\oplus M_i,\Delta_M(j)) \to \Hom_A(K(i),\Delta_M(j))\to 0.$$

(a) Let $i\geq j.$ Since $U(i)\in \F(\Delta_{M}(r) : r>i)$ and $\Hom_A( \Delta_{M}(r),\Delta_{M}(j))=0$ $\forall\,r>i,$ we can conclude that 
$\Hom_A(U(i),\Delta_M(j))=0.$ On the other hand, by using that $$K(i)\in\Fac\left(\bigoplus_{l>i} M_l\right)\subseteq \Fac\left(\bigoplus_{l>j} M_l\right),$$  it follows that 
$\Hom_A(K(i),\Delta_M(j))=0.$
\

(b) Take $i=j$ in $(*).$ Then, from (a),  $\Hom_A(K(i),\Delta_M(i))=0$ and $\Hom_A(U(i),\Delta_M(i))=0.$ Hence, by $(*)$ we get that $\Hom_A(Q(i),\Delta_M(i))\simeq\Hom_A(M_i,\Delta_M(i)).$ Finally, the isomorphism 
$\Hom_A(Q(i), \Delta_M(i))\simeq \End_A(\Delta_M(i))$ follows from Lemma \ref{LCMTheta}.
\

(c) It follows by applying $\dim_k$ on the exact sequence $(*).$
\

(d) Since $\tau M\simeq\bigoplus_{r=1}^t \tau M_r,$ it is enough to show that $\Hom_A(\Delta_M(j),\tau(M))=0$ $\forall\, j\in[1,t].$ But, the aforementioned statement holds true since $\Delta_M(j)\in \Fac(M_j)\subseteq \Fac(M)$ and $M$ is $\tau$-rigid. Finally, we have
$\Ext^1_A(M_i,\Delta_M(j))\simeq D\overline{\Hom}_A(\Delta_M(j),\tau M_i)=0.$
\end{proof}

Now, we are ready to announce and prove our first main result.

\begin{theorem}\label{MTM} For a basic and $\tau$-rigid module $M\in\modu\,(A)$ with a TF-admissible decomposition $M= \bigoplus_{i=1}^t M_i,$ and  the Ext-projective stratifying system $(\Delta_M, \underline{Q}, \leq)$ induced by $(\Delta_M,\leq),$ the following statements hold true.
\begin{itemize}
\item[$\mathrm{(a)}$] The matrix $(G^M)^{tr}\,S^M$ is upper triangular and  $[(G^M)^{tr}\,S^M]_{i,j}= \dim_k \Hom_A(M_i,\Delta_M(j))$ $\forall\,i,j.$
\item[$\mathrm{(b)}$] $R^M$ is an upper triangular matrix with zeros in its diagonal. Moreover, we have that 
\begin{center}
$R^M=0$ $\Leftrightarrow$ $\Hom_A(M_i,\Delta_M(j))\simeq \Hom_A(Q(i),\Delta_M(j))$ $\forall\,i<j.$
\end{center}
\item[$\mathrm{(c)}$] $C_{\Delta_M}=(G^M)^{tr}\,S^M+R^M.$
\item[$\mathrm{(d)}$]  $R^M=0$ if $M\in\F(\Delta_M).$ 
\item[$\mathrm{(e)}$]  $|G_{\Delta_M}| = \prod_{i=1}^t [(G^M)^{tr}\,S^M]_{i,i}.$ 
\end{itemize}
\end{theorem}

\begin{proof} 
Let $i,j\in[1,t].$ 
Then, by Theorem \ref{formula} we get 
$$[(G^M)^{tr}\,S^M]_{i,j}=\langle g^{M_i},\udim(\Delta_M(j))\rangle=\dim_k(\Hom_A(M_i,\Delta_M(j)))-\dim_k(\Hom_A(\Delta_M(j),\tau_A M_i)).$$
On the other hand, from Lemma \ref{lem:M-theta} (d), we know that $\Hom_A(\Delta_M(j),\tau_A M_i)=0,$ proving that 
$$[(G^M)^{tr}\,S^M]_{i,j}=\dim_k(\Hom_A(M_i,\Delta_M(j))).$$
Thus, from Lemma \ref{lem:M-theta} (c), we also get the equality $C_{\Delta_M}=(G^M)^{tr}\,S^M+R^M,$ proving (c). Moreover, by 
Lemma \ref{lem:M-theta} (a)  and the exact sequence $(*)$ in the proof of Lemma \ref{lem:M-theta}, we can obtain (b). 
Also, since $C_{\Delta_M}$ is upper triangular (see Lemma \ref{LCMTheta}), then by (b) and (c) it follows that the matrix $(G^M)^{tr}\,S^M$ is upper triangular. Finally, (e) follows directly from (b), (c) and Lemma~\ref{LCMTheta}.
\

Assume now that $M\in\F(\Delta_M).$ Then, by the proof of \cite[Corollary 3.8]{Mendoza2019}, we get that the exact sequences $\eta_i$ and 
$\varepsilon_i$ are isomorphic, for each $i\in[1,t].$ Therefore we obtain (d).
\end{proof}

As we have seen in Theorem \ref{MTM} (d), the condition $M\in\F(\Delta_M)$ is enough to get that $R^M=0.$ 
In the following, we characterise when 
$M\in\F(\Delta_M).$ 

\begin{proposition}\label{MinFD} Let $M\in\modu\,(A)$ be a basic and $\tau$-rigid module  with a TF-admissible decomposition $M= \bigoplus_{i=1}^t M_i,$ and let 
$(\Delta_M, \underline{Q}, \leq)$ be the Ext-projective stratifying system  induced by $(\Delta_M,\leq).$ Then, for any $i\in[1,t],$ the following statements are equivalent.
\begin{itemize}
\item[(a)] $M_i\in\F(\Delta_M).$
\item[(b)] The exact sequences $\eta_i$ and $\varepsilon_i$ are isomorphic.
\item[(c)] $Q(i)\simeq M_i.$ 
\item[(d)] $\Ext^1_A(Q(i), K(i))=0.$ 
\item[(e)] $K(i)\simeq U(i).$ 
\end{itemize}
\end{proposition}
\begin{proof} (a) $\Rightarrow$ (b) It follows by the proof of \cite[Corollary 3.8]{Mendoza2019}.
\

(b) $\Rightarrow$ (c) It is trivial.
\

(c) $\Rightarrow$ (d) It follows from $\Ext^1_A(M,\Fac(M))=0$ since $K(i)\in\Fac(M).$ 
\

(d) $\Rightarrow$ (e)  Since $\Ext^1_A(Q(i), K(i))=0,$ there is some $q: Q(i)\to M_i$ with $\pi_iq=\beta_i.$ On the other hand 
by the proof of Lemma \ref{lem:M-theta}, there is some $p:M_i\to Q(i)$ with $\beta_ip=\pi_i.$ Since $\beta_i(pq)=\beta_i$ and 
$\beta_i:Q(i)\to \Delta_M(i)$ is right minimal \cite[Lemma 2.3]{Marcos2005a}, we get that $pq$ is an isomorphism, and thus 
$q:Q(i)\xrightarrow{\thicksim} M_i$ since $Q(i)$ and $M_i$ are indecomposable. Thus, we have that the exact sequences $\eta_i$ and 
$\varepsilon_i$ are isomorphic. In particular, $K(i)\simeq U(i).$ 
\

(e) $\Rightarrow$ (a) By the exact sequence $\varepsilon_i$ and the fact that $\F(\Delta_M)$ is closed under extensions, 
we conclude that $M_i\in\F(\Delta_M).$
\end{proof}

\begin{remark}
Recall from \cite{King1994} that given $M\in\modu(A)$ and a vector $v\in \mathbb{R}^n,$ we say that $M$ is $v$-semistable if $\langle v, \udim M \rangle =0$ and $\langle v, \udim L \rangle \leq 0$ for all submodule $L$ of $M$.
Using this notion of stability, one can construct (for every finite dimesional algebra $A$) a geometric invariant known as the \textit{wall-and-chamber structure} of $A$. 
It was shown in \cite{BST2019} that the wall-and-chamber structure of an algebra $A$ can be described by its $\tau$-tilting theory, via the $g$-vectors of the indecomposable $\tau$-rigid objects.

Now, let $M\in\modu(A)$ be basic and $\tau$-rigid, with a TF-admissible decomposition 
$M= \bigoplus_{i=1}^t M_i,$ and let $\Delta_M$ be its associated stratifying system.
From the construction of $\Delta_M$ and the results of \cite{BST2019}, it follows that the $A$-module $\Delta_M(i)$ is $g^{M_j}$-semistable,  for all $1 \leq i \leq t$ and $j > i$.
It would be interesting to see if the Cartan matrix $C_{\Delta_M}$ has a geometric realisation in the wall-and-chamber structure of $A$.
\end{remark}

Let $\Lambda$ be a finite dimensional $k$-algebra. We recall from \cite{MMS2019} that the Cartan group $G_\Lambda$ of $\Lambda$ is the Cokernel of the Cartan map $C_\Lambda:K_0(\proj(\Lambda))\to K_0(\Lambda),\;[P_i]\mapsto\udim(P_i),$ where $\{P_i\}_{i=1}^n$ is a representative set of pairwise non-isomorphic indecomposable projective $\Lambda$-modules.

\begin{corollary} 
Let $M\in\modu\,(A)$ be basic and $\tau$-tilting and consider a TF-admissible decomposition $M= \bigoplus_{i=1}^t M_i$ such that 
$M\in\F(\Delta_M).$  Then, for the algebra $B:=\End_A(M)^{op},$ we have that $G_B\simeq \CoKer(S^M).$
\end{corollary}
\begin{proof} Let $(\Delta_M, \underline{Q}, \leq)$ be the Ext-projective stratifying system induced by $(\Delta_M,\leq).$ By Theorem \ref{MTM} and Proposition \ref{MinFD}, we get that $C_{\Delta_M}=(G^M)^{tr}\,S^M$ and $M_i\simeq Q(i)$ $\forall\,i\in[1,t].$ 
On the other hand, by \cite[Theorem~5.1]{AIR}, we conclude that the 
matrix $G^M\in\Mat_{n\times n}(\mathbb{Z})$  induces an isomorphism $\mathbb{Z}^n\to\mathbb{Z}^n,\;X\mapsto G^M\,X,$ of 
abelian groups. Therefore, $\CoKer(C_{\Delta_M})\simeq\CoKer(S^M)$ and then the result follows from \cite[Theorem 4.2]{MMS2019}.
\end{proof}

As we have already seen, the Cartan matrix $C_{\Theta}$ associated to a stratifying system $\Theta$ is an upper-triangular matrix. 
We now will characterise in Theorem~\ref{thm:diagonal} the stratifying systems coming from $\tau$-tilting modules $M$ whose associated Cartan matrix $C_{\Theta}$ is diagonal. 
Before proving our second main theorem, we need the following intermediate result.

\begin{proposition}\label{PMdiag} For a basic and $\tau$-rigid module $M\in\modu\,(A),$ with a TF-admissible decomposition $M= \bigoplus_{i=1}^t M_i,$  the following statements (a), (b) and (c) are equivalent.
\begin{itemize}
\item[$\mathrm{(a)}$] The matrix $(G^M)^{tr}\,S^M$ is diagonal and $\Hom_A(M_i, K(j))=0$ for $i<j.$
\item[$\mathrm{(b)}$] $\Hom_A(M_i, M_j)=0$ for $i<j.$
\item[$\mathrm{(c)}$] For $i<j,$ we have that $\Hom_A(\Delta_M(i), \Delta_M(j))=0,$ $\Hom_A(M_i, K(j))=0$ and the map
$$\Hom_A(K(i),\Delta_M(j))\to\Ext^1_A(\Delta_M(i), \Delta_M(j)),\;f\mapsto f\cdot\varepsilon_i$$
is an isomorphism.
\end{itemize}
Moreover, for $B:=\End_A(M)^{op}$ and $d_i:=\dim_k\End_A(\Delta_M(i)),$  we have that  
$G_B\simeq\bigoplus_{i=1}^t\,\mathbb{Z}/d_i\mathbb{Z}$ if $M\in\F(\Delta_M)$ and $(G^M)^{tr}\,S^M$ is diagonal.
\end{proposition}
\begin{proof}  Let $i,j\in[1,t].$ Consider the canonical exact sequence 
$$\varepsilon_j:\;0\to K(j)\to M_j\to \Delta_M(j)\to 0$$
of $M_j$ with respect to the torsion pair $\left(\Fac(\bigoplus_{l>j}M_l), (\bigoplus_{l>j}M_l)^\perp\right).$ Note that $\varepsilon_j$ lies in $\Fac(M);$ and since $\Ext^1_A(M,\Fac(M))=0,$ from $\varepsilon_j$ we get the exact sequence
$$(*):\; 0\to \Hom_A(M_i,K(j))\to \Hom_A(M_i, M_j)\to \Hom_A(M_i, \Delta_M(j))\to 0.$$

(a) $\Leftrightarrow$ (b): It follows from (*) and Theorem \ref{MTM} (a).
\

(a) $\Leftrightarrow$ (c): Let $i<j.$ Consider the exact sequence $\varepsilon_i:\;0\to K(i)\to M_i\to \Delta_M(i)\to 0.$ Then, we get the exact sequence
{\small $$0\to\Hom_A(\Delta_M(i), \Delta_M(j))\to \Hom_A(M_i,   \Delta_M(j))\to \Hom_A(K(i), \Delta_M(j))\to \Ext^1_A(\Delta_M(i), \Delta_M(j))\to 0 $$}
Since $\Ext^1_A(M_i,   \Delta_M(j))=0$ (see Lemma \ref{lem:M-theta} (d)). Thus, the equivalence between (a) and (c) follows from the above exact sequence and Theorem \ref{MTM} (a).
\

Let $M\in\F(\Delta_M)$ and $(G^M)^{tr}\,S^M$ be a diagonal matrix. 
By Theorem \ref{MTM} and Proposition \ref{MinFD}, we get that \mbox{$C_{\Delta_M}=(G^M)^{tr}\,S^M$} and $M_i\simeq Q(i)$ $\forall\,i.$ Then, by \cite[Theorem 4.2]{MMS2019} and Lemma \ref{lem:M-theta} (b), we conclude that  
$G_B\simeq\bigoplus_{i=1}^t\,\mathbb{Z}/d_i\mathbb{Z}.$
\end{proof}

\begin{theorem}\label{thm:diagonal}
For a basic $\tau$-tilting module $M\in\modu\,(A),$ with a TF-admissible decomposition $M= \bigoplus_{i=1}^n M_i$ such that 
$M\in\F(\Delta_M),$  the following statements are equivalent.
\begin{itemize}
\item[$\mathrm{(a)}$] The matrix $C_{\Delta_M}$ is diagonal.
\item[$\mathrm{(b)}$] $\Hom_A(M_i, M_j)=0$ for $i<j.$
\item[$\mathrm{(c)}$] $\Hom_A(M_i, \Delta_M(j))=0$ for $i<j.$
\item[$\mathrm{(d)}$] For $i<j,$ we have that $\Hom_A(\Delta_M(i), \Delta_M(j))=0$ and the map
$$\Hom_A(K(i),\Delta_M(j))\to\Ext^1_A(\Delta_M(i), \Delta_M(j)),\;f\mapsto f\cdot\eta_i$$
is an isomorphism.
\end{itemize}
Moreover, if one of the above conditions holds true, then $M\simeq {}_AA,$  $A$ is a weakly triangular algebra and $\Hom_A(M_i, K(j))=0$ for 
$i<j.$
\end{theorem}

\begin{proof} 
Let $(\Delta_M, \underline{Q}, \leq)$ be the Ext-projective stratifying system induced by $(\Delta_M,\leq).$ By Theorem \ref{MTM} and Proposition \ref{MinFD}, we get that $C_{\Delta_M}=(G^M)^{tr}\,S^M,$ $U(i)\simeq K(i)$ and $M_i\simeq Q(i)$ $\forall\,i.$ Thus, the fact that $\Hom_A(M_i, K(j))=0$ for $i<j$ (by assuming one of the above conditions) and the equivalences between (a), (b), (c) and (d) follow from Proposition \ref{PEDiagCtheta}.
\

Now, assume that one of the above equivalent conditions holds true. 
Let us proof, firstly, that  $M\simeq {}_AA.$ Indeed, since the decomposition $M= \bigoplus_{i=1}^n M_i$ is TF-admissible, we get from (b) that $M_i\not\in\Fac(\bigoplus_{j\neq i}M_j)$ $\forall\,i\in[1,n].$
 On the other hand, it is known that $\bigoplus_{j\neq i}M_j$ admits two torsion classes, namely, $\Fac(\bigoplus_{j\neq i}M_j)$ and ${}^\perp(\bigoplus_{j\neq i}\tau M_j).$ 
Moreover, since $\Fac(\bigoplus_{j\neq i}M_j)\subsetneq{}^\perp(\bigoplus_{j\neq i}\tau M_j)$ and $M_i\not\in\Fac(\bigoplus_{j\neq i}M_j),$ we have that  $M_i\in {}^\perp(\bigoplus_{j\neq i}\tau M_j)$ and thus the mutation of $M= \bigoplus_{i=1}^t M_i$ at $M_i$ is descendent, for all $i\in[1,n].$ 
Let us show that $\Fac(M)=\modu(A).$ 
Suppose that $\Fac(M)\neq \modu(A).$ Since $(\Fac(M), M^\perp)$ is a torsion pair in $\modu(A),$ it follows from \cite[Theorem 3.1]{DIJ}, that there exists a mutation $M'$ of $M$ such that $\Fac(M)\subsetneq\Fac(M')$ contradicting that any mutation of $M$ is descendent. 
Therefore $\Fac(M)=\modu(A)$ and then $M\simeq {}_AA.$ 
In particular, $A$ is a standardly stratified algebra since  $M_i\simeq Q(i)$ $\forall\,i;$ and by \cite[Remark 4.5]{MMS2019}, we conclude that $A$ is a weakly triangular algebra.
\end{proof}

\section{Examples}\label{Examples}

We finish the paper by calculating the Cartan matrix of stratifying systems induced by $\tau$-rigid modules. 
But before we go to the examples, we recall the notation of Lemma~\ref{lem:M-theta} of the two short exact sequences naturally associated to $\Delta_M(i)$ for $i\in \{1,\dots,t\}$.
$$\varepsilon_i:\; 0\to  K(i) \xrightarrow{\iota_i} M_i \xrightarrow{\pi_i} \Delta_M(i) \rightarrow 0,$$
$$\eta_i:\; 0\to U(i)\xrightarrow{\alpha_i} Q(i)\xrightarrow{\beta_i}\Delta_M(i)\to 0.$$

\begin{example}
Let $A$ be the quotient path $k$-algebra given by the quiver 
$$
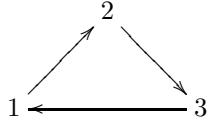
\xymatrix{
  & 2\ar[dr]& \\
  1\ar[ru] & & 3\ar[ll] }$$
and the third power of the ideal generated by all the arrows.
The Auslander-Reiten quiver of $A$ can be seen in Figure \ref{fig:Ar-quiverA}.
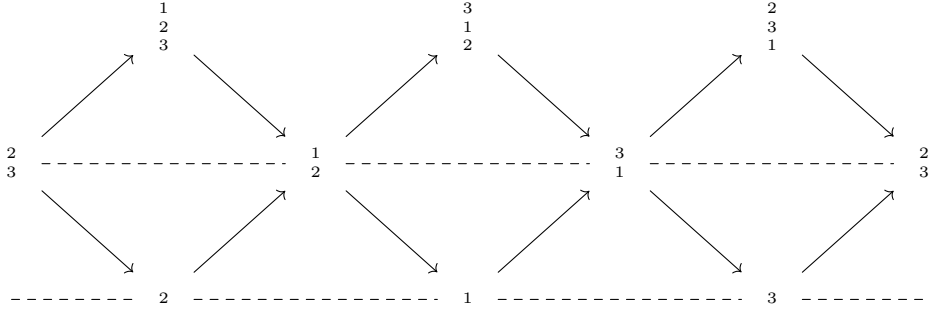
\begin{figure}[h]
    \centering
			\begin{tikzpicture}[line cap=round,line join=round ,x=2.0cm,y=1.8cm]
				\clip(-2.2,-0.1) rectangle (4.1,2.5);
					\draw [->] (-0.8,0.2) -- (-0.2,0.8);
					\draw [->] (1.2,0.2) -- (1.8,0.8);
					\draw [->] (3.2,0.2) -- (3.8,0.8);
					\draw [->] (-1.8,1.2) -- (-1.2,1.8);
					\draw [->] (0.2,1.2) -- (0.8,1.8);
					\draw [->] (2.2,1.2) -- (2.8,1.8);
					\draw [dashed] (-0.8,0.0) -- (0.8,0.0);
					\draw [dashed] (1.2,0.0) -- (2.8,0.0);
					\draw [dashed] (-1.8,1.0) -- (-0.2,1.0);
					\draw [dashed] (0.2,1.0) -- (1.8,1.0);
					\draw [dashed] (2.2,1.0) -- (3.8,1.0);
					\draw [dashed] (-2.0,0.0) -- (-1.2,0.0);
					\draw [dashed] (3.2,0.0) -- (4.0,0.0);
					\draw [->] (0.2,0.8) -- (0.8,0.2);
					\draw [->] (2.2,0.8) -- (2.8,0.2);
					\draw [->] (-1.8,0.8) -- (-1.2,0.2);
					\draw [->] (-0.8,1.8) -- (-0.2,1.2);
					\draw [->] (1.2,1.8) -- (1.8,1.2);
					\draw [->] (3.2,1.8) -- (3.8,1.2);
				
				\begin{scriptsize}
					\draw[color=black] (-1,0) node {$\rep{2}$};
					\draw[color=black] (1,0) node {$\rep{1}$};
					\draw[color=black] (3,0) node {$\rep{3}$};
					\draw[color=black] (-2,1) node {$\rep{2\\3}$};
					\draw[color=black] (0,1) node {$\rep{1\\2}$};
					\draw[color=black] (2,1) node {$\rep{3\\1}$};
					\draw[color=black] (4,1) node {$\rep{2\\3}$};
					\draw[color=black] (-1,2) node {$\rep{1\\2\\3}$};
					\draw[color=black] (1,2) node {$\rep{3\\1\\2}$};
					\draw[color=black] (3,2) node {$\rep{2\\3\\1}$};
				\end{scriptsize}
			\end{tikzpicture}
\caption{The Auslander-Reiten quiver of $A$}
    \label{fig:Ar-quiverA}
\end{figure}
Note that every module is represented by its Loewy series and both copies of $\rep{2\\3}$ should be identified. Thus,  the Auslander-Reiten 
quiver of $A$ has the shape of a cylinder. 

Now take the $\tau$-rigid modules $A = \rep{1\\2\\3} \oplus \rep{2\\3\\1} \oplus \rep{3\\1\\2}$ and $M= \rep{1\\2\\3}\oplus \rep{1\\2} \oplus \rep{1}.$ It can be shown that, in both cases, the $\tau$-tilting modules are written in a TF-admissible fashion. Hence, 
their corresponding induced stratifying system are $\Delta_A = \left\{ \rep{1}, \rep{2}, \rep{3\\1\\2} \right\}$ and $\Delta_M= \left\{ \rep{1\\2\\3}, \rep{1}, \rep{2} \right\}$.

\begin{table}[h]
\centering
\begin{tabular}{|c|c|c|c|c|c|}\hline
 $i$ & $\Delta_A(i)$ & $M_i$  & $K(i)$ & $Q(i)$ & $U(i)$\\\hline
 $1$ & $\rep{1}$ & $\rep{1\\2\\3}$  & $\rep{2\\3}$ & $\rep{1\\2}$ & $\rep{2}$\\\hline
 $2$ & $\rep{2}$ & $\rep{2\\3\\1}$  & $\rep{3\\1}$ & $\rep{2}$ & $\rep{0}$\\\hline
  $3$ & $\rep{3\\1\\2}$ & $\rep{3\\1\\2}$  & $\rep{0}$ & $\rep{3\\1\\2}$ & $\rep{0}$\\\hline
\end{tabular}
\vspace{0.2cm}
\caption{Modules associated to the stratifying system $\Delta_A$}
\label{table:Delta_A}
\end{table}

\begin{table}[h]
\centering
\begin{tabular}{|c|c|c|c|c|c|}\hline
 $i$ & $\Delta_M(i)$ & $M_i$  & $K(i)$ & $Q(i)$ & $U(i)$\\\hline
 $1$ & $\rep{1\\2\\3}$ & $\rep{1\\2\\3}$  & $\rep{0}$ & $\rep{1\\2\\3}$ & $\rep{0}$\\\hline
 $2$ & $\rep{1}$ & $\rep{1\\2}$  & $\rep{2}$ & $\rep{1\\2}$ & $\rep{1}$\\\hline
  $3$ & $\rep{2}$ & $\rep{2}$  & $\rep{0}$ & $\rep{2}$ & $\rep{0}$\\\hline
\end{tabular}
\vspace{0.2cm}
\caption{Modules associated to the stratifying system $\Delta_M$}
\label{table:Delta_M}
\end{table}

\noindent
Using the modules in Table~\ref{table:Delta_A} and Table~\ref{table:Delta_M} and the Auslander-Reiten quiver of $A,$  one can calculate easily the following matrices.
$$
C_{\Delta_A}=\left(\begin{matrix}
1&0&1\\
0&1&1\\
0&0&1
\end{matrix}\right)
\qquad
G^A=\left(\begin{matrix}
1&0&0\\
0&1&0\\
0&0&1
\end{matrix}\right)
\qquad
S^A=\left(\begin{matrix}
1&0&1\\
0&1&1\\
0&0&1
\end{matrix}\right)
\qquad
R^A=\left(\begin{matrix}
0&0&0\\
0&0&0\\
0&0&0
\end{matrix}\right).
$$

$$
C_{\Delta_M}=\left(\begin{matrix}
1&1&0\\
0&1&0\\
0&0&1
\end{matrix}\right)
\qquad
G^M=\left(\begin{matrix}
1&0&0\\
0&1&-1\\
0&1& -1
\end{matrix}\right)
\qquad
S^M=\left(\begin{matrix}
1&1&0\\
1&0&1\\
1&0&0
\end{matrix}\right)
\qquad
R^M=\left(\begin{matrix}
0&0&0\\
0&0&0\\
0&0&0
\end{matrix}\right).
$$
In particular, it is straightforward to verify that $C_{\Delta_A} = (G^A)^{tr} S^A + R^A$ and $C_{\Delta_M} = (G^M)^{tr} S_M + R^M$, as shown in Theorem~\ref{MTM}.
\end{example}

\begin{example} Let $A$ be the quotient path $k$-algebra given by the quiver 
$$\xymatrix{
  3 \ar@<1ex>[r]^{\beta_1}  &  1 \ar@<1ex>[l]^{\alpha_1}\ar@<1ex>[r]^{\beta_2} & 2 \ar@<1ex>[l]^{\alpha_2} }$$
and the ideal generated by the set of relations $\{\beta_1\alpha_1\beta_1,$ $\beta_1\alpha_1-\alpha_2\beta_2,$ $\alpha_2\beta_2\alpha_2\}.$
Now take the $\tau$-rigid module $A = \rep{\;1\;\\3 \;\; 2\\ \;1\;} \oplus \rep{2\\1\\2} \oplus \rep{3\\1\\3}.$ Hence, 
its corresponding induced stratifying system are $\Delta_A = \left\{ \rep{1}, \rep{2\\1\\2}, \rep{3\\1\\3} \right\}.$ 
A quick calculation shows that the modules associated to the stratifying system $\Delta_A$ are the following. 

\begin{table}[h]
\centering
\begin{tabular}{|c|c|c|c|c|c|}\hline
 $i$ & $\Delta_A(i)$ & $M_i$  & $K(i)$ & $Q(i)$ & $U(i)$\\\hline
 $1$ & $\rep{1}$ & $\rep{\;1\;\\3 \;\; 2\\ \;1\;}$  & $\rep{3 \;\; 2\\ \;1\;}$ & $\rep{1}$ & $\rep{0}$\\\hline
 $2$ & $\rep{2\\1\\2}$ & $\rep{2\\1\\2}$  & $\rep{0}$ & $\rep{2\\1\\2}$ & $\rep{0}$\\\hline
  $3$ & $\rep{3\\1\\3}$ & $\rep{3\\1\\3}$  & $\rep{0}$ & $\rep{3\\1\\3}$ & $\rep{0}$\\\hline
\end{tabular}
\vspace{0.2cm}
\caption{Modules associated to the stratifying system $\Delta_A$}
\label{table:Delta_A3}
\end{table}

\noindent
From Table~\ref{table:Delta_A3} we can calculate the matrices associated to $\Delta_{A}$, which are below. 
In particular, it is straightforward to verify that $C_{\Delta_A} = (G^A)^{tr} S^A + R^A,$  as shown in Theorem~\ref{MTM}.
$$
C_{\Delta_A}=\left(\begin{matrix}
1&0&0\\
0&2&0\\
0&0&2
\end{matrix}\right)
\qquad
G^A=\left(\begin{matrix}
1&0&0\\
0&1&0\\
0&0&1
\end{matrix}\right)
\qquad
S^A=\left(\begin{matrix}
1&1&1\\
0&2&0\\
0&0&2
\end{matrix}\right)
\qquad
R^A=\left(\begin{matrix}
0&-1&-1\\
0&0&0\\
0&0&0
\end{matrix}\right).
$$
\end{example}

\footnotesize

\vskip3mm \noindent Octavio Mendoza:\\
Instituto de Matem\'aticas,\\
Universidad Nacional Aut\'onoma de M\'exico,\\
Circuito Exterior, Ciudad Universitaria,\\
M\'exico D.F. 04510, M\'EXICO.

{\tt omendoza@matem.unam.mx}

\vskip3mm \noindent Corina S\'aenz:\\
Departamento de Matem\'aticas, Facultad de Ciencias,\\
Universidad Nacional Aut\'onoma de M\'exico,\\
Circuito Exterior, Ciudad Universitaria,\\
M\'exico D.F. 04510, M\'EXICO.

{\tt ecsv@ciencias.unam.mx}

\vskip3mm \noindent Hipolito Treffinger:\\
Universit\'e de Paris et Sorbonne Universit\'e.\\
CNRS, IMJ-PRG,\\
B\^atiment Sophie Germain,\\
5 rue Thomas Mann,\\
75205 Paris Cedex 13,
FRANCE

{\tt treffinger@imj-prg.fr}

\end{document}